\documentclass[11pt]{article}

\usepackage{amsthm}
\usepackage{amsmath}
\usepackage{amssymb}

\usepackage{graphicx}
\usepackage{tikz}

\usepackage{enumerate}

\newtheorem{theorem}{Theorem}[section]
\newtheorem{proposition}[theorem]{Proposition}
\newtheorem{lemma}[theorem]{Lemma}

\newcommand{\E}{{\mathbb{E}}}
\newcommand{\Prob}{{\mathbb{P}}}
\newcommand{\disc}{{\textup{disc}}}
\newcommand{\discP}{{\textup{disc}_P}}
\newcommand{\codeg}{{\textup{codeg}}}
\newcommand{\Bin}{{\textup{Bin}}}
\newcommand{\RandHyp}[3]{{\mathcal{H}_{#1}(#2,#3)}}

\newcommand{\ignore}[1]{}

\title{Discrepancy of random graphs and hypergraphs}
\author{
Jie Ma\thanks{Department of Mathematics, UCLA, Los Angeles, CA 90095. Email: {\tt jiema@math.ucla.edu}.
Research supported in part by AMS-Simons travel grant.}
\and Humberto Naves\thanks{Department of Mathematics, UCLA, Los Angeles, CA 90095. Email: {\tt hnaves@math.ucla.edu}.}
\and Benny Sudakov\thanks{Department of Mathematics, UCLA, Los Angeles, CA 90095. Email: {\tt bsudakov@math.ucla.edu}.
Research supported in part by NSF grant DMS-1101185, by AFOSR MURI grant FA9550-10-1-0569 and by a USA-Israel
BSF grant.}}

\date{}

\AtBeginDocument{\setlength{\abovedisplayskip}{7pt}}
\AtBeginDocument{\setlength{\belowdisplayskip}{7pt}}

\begin{document}
\maketitle
\setcounter{page}{1}

\vspace{-2em}
\begin{abstract}
Answering in a strong form a question posed by Bollob\'as and Scott,  in this paper we determine the discrepancy
between two random $k$-uniform hypergraphs, up to a constant factor depending solely on $k$.
\end{abstract}

\section{Introduction}
A {\em hypergraph} $H$ is an ordered pair $H=(V,E)$, where $V$ is a finite set (the {\em vertex set}), and $E$ is
a family of distinct subsets of $V$ ({\em the edge set}). The hypergraph $H$ is $k$-uniform if all its edges
are of size $k$. In this paper we consider only $k$-uniform hypergraphs. The \textit{edge density} of a
$k$-uniform hypergraph $H$ with $n$ vertices is $\rho_H=e(H)/\binom{n}{k}$. We define the \textit{discrepancy} of
$H$ to be
\begin{equation}\label{disc_H}
\disc(H)=\max_{S\subseteq V(H)} \left|e(S)- \rho_H \binom{|S|}{k} \right|,
\end{equation}
where $e(S)=e(H[S])$ is the number of edges in the sub-hypergraph induced by $S$. The discrepancy can be viewed as a
measure of how uniformly the edges of $H$ are distributed among the vertices. This important concept appears naturally
in various branches of combinatorics and has been studied by many researchers in recent years. The discrepancy is closely
related to the theory of quasi-random graphs (see \cite{chung_quasi}), as the property $\disc(G)=o(|V(G)|^2)$ implies
the quasi-randomness of the graph $G$.

Erd\H{o}s and Spencer \cite{erdos_imb} proved that for $k\ge 2$, any $k$-uniform hypergraph $H$ with $n$ vertices has
a subset $S$ satisfying $\left|e(S)-\frac{1}{2}\binom{|S|}{k}\right|\ge cn^{\frac{k+1}{2}}$, which implies the bound
$\disc(H)\ge cn^{\frac{k+1}{2}}$ for $k$-uniform hypergraphs $H$ of edge density $\frac{1}{2}$. Erd\H{o}s, Goldberg,
Pach and Spencer \cite{erdos_cutting}  obtained a similar lower bound for graphs of edge density smaller than
$\frac{1}{2}$. These results were later generalized by Bollob\'as and Scott in \cite{bollobas_disc}, who proved the
inequality $\disc(H)\ge c_k\sqrt{r}n^{\frac{k+1}{2}}$ for $k$-uniform hypergraphs $H$, whenever $r=\rho_H(1-\rho_H)\ge 1/n$. The
random hypergraphs show that all the aforementioned lower bounds are optimal up to constant factors. For more
discussion and general accounts of discrepancy, we refer the interested reader to Beck and S\'os \cite{beck_disc},
Bollob\'as and Scott \cite{bollobas_disc}, Chazelle \cite{chaz_the}, Matou\v sek \cite{mato_geom} and S\'os
\cite{sos_irre}.

A similar notion is the relative discrepancy of two hypergraphs. Let $G$ and $H$ be two $k$-uniform
hypergraphs over the same vertex set $V$, with $|V| = n$. For a bijection $\pi:V \to V$, let $G_\pi$ be obtained from
$G$ by permuting all edges according to $\pi$, i.e., $E(G_\pi) = \pi(E(G))$. The \textit{overlap} of $G$ and $H$ with
respect to $\pi$, denoted by $G_\pi\cap H$, is a hypergraph with the same vertex set $V$ and with edge set
$E(G_\pi)\cap E(H)$. The \textit {discrepancy of $G$ with respect to $H$} is
\begin{equation}\label{disc_GH} \disc(G,H) = \max_\pi
\left|e(G_\pi \cap H) - \rho_G\rho_H\binom{n}{k}\right|,
\end{equation}
where the maximum is taken over all bijections $\pi:V \to V$. For random bijections $\pi$, the expected size of $E(G_\pi)\cap E(H)$
is $\rho_G\rho_H\binom{n}{k}$, thus $\disc(G,H)$ measures how much the overlap can deviate from its
average. In a certain sense, the definition \eqref{disc_GH} is more general than \eqref{disc_H}, because one can write
$\disc(H)=\max_{1\le i\le n} \disc(G_i, H)$, where $G_i$ is obtained from the complete $i$-vertex $k$-uniform
hypergraph by adding $n-i$ isolated vertices.

Bollob\'as and Scott introduced the notion of relative discrepancy in \cite{bollobas_inter} and showed that for
any two $n$-vertex graphs $G$ and $H$, if $\frac{16}{n}\le \rho_G,\rho_H\le 1-\frac{16}{n}$, then
$\disc(G,H)\ge c\cdot f(\rho_G,\rho_H)\cdot n^{\frac{3}{2}}$, where $c$ is an absolute constant and
$f(x,y)=x^2(1-x)^2y^2(1-y)^2$. As a corollary, they proved a conjecture in \cite{erdos_cutting} regarding the
{\it bipartite discrepancy} $\disc(G,K_{\lfloor\frac{n}{2}\rfloor, \lceil\frac{n}{2} \rceil})$. Moreover, they also
conjectured that a similar bound holds for $k$-uniform hypergraphs, namely, there exists $c=c(k,\rho_G, \rho_H)$ for
which $\disc(G,H) \ge cn^{\frac{k+1}{2}}$ holds for any $k$-uniform hypergraphs $G$ and $H$ satisfying
$\frac{1}{n} \le \rho_G, \rho_H \le 1 - \frac{1}{n}$.

In their paper, Bollob\'as and Scott also asked the following question (see Problem 12 in \cite{bollobas_inter}). Given
two random $n$-vertex graphs $G,H$ with constant edge probability $p$, what is the expected value of $\disc(G,H)$?
In this paper, we solve this question completely for general $k$-uniform hypergraphs. Let $\RandHyp{k}{n}{p}$ denote
the random $k$-uniform hypergraph on $n$ vertices, in which every edge is included independently with probability $p$.
We say that an event happens \textit{with high probability}, or w.h.p. for brevity, if it happens with probability at
least $1 - n^{-w(n)}$, where here and later $w(n)>0$ denotes an arbitrary function tending to infinity together with $n$.
\begin{theorem}
\label{theorem_main}
For positive integers $n$ and $k$, let $N={n-\frac{n}{k} \choose k-1}$. Let $G$ and $H$ be
two random hypergraphs distributed according to $\RandHyp{k}{n}{p}$ and $\RandHyp{k}{n}{q}$ respectively, where
$\frac{w(n)}{N} \le p \le q \le \frac12$.
\begin{enumerate}
\item [(1)] {\small \sc dense case} -- If $pqN > \frac{1}{30}\log n$, then w.h.p.
  $\disc(G,H)= \Theta_k\left(\sqrt{pq\binom{n}{k}n\log n} \right)$;
\item [(2)] {\small \sc sparse case} -- If $pq N\le \frac{1}{30} \log n$, let
  $\gamma = \frac{\log n}{pq N}$, then
\begin{enumerate}
\item [(2.1)] if $pN \ge \frac{\log n}{5\log \gamma}$, then w.h.p.
  $\disc(G,H)= \Theta_k\left(\frac{n \log n}{\log \gamma} \right)$.
\item [(2.2)] if $pN< \frac{\log n}{5\log \gamma}$, then w.h.p.
  $\disc(G,H)= \Theta_k\left(p\binom{n}{k} \right)$.
\end{enumerate}
\end{enumerate}
\end{theorem}

The previous theorem also provides tight bounds when $p$ and/or $q  \ge \frac12$, as we shall
see in the concluding remarks. The result of Theorem~\ref{theorem_main} in the sparse range is
closely related to the recent work of the third author with Lee and Loh \cite{sudakov_similarity}. Among other
results, the authors of \cite{sudakov_similarity} show that two independent copies $G, H$ of the random graph $G(n,p)$
with $p \ll \sqrt{\log n/n}$ w.h.p. have overlap of order $\Theta\left(n\frac{\log n}{\log \gamma}\right)$, where
$\gamma= \frac{\log n}{p^2n}$. Hence $\disc(G,H)=\Theta\left(n\frac{\log n}{\log \gamma}\right)$ holds, since in
this range of edge probability, $n\frac{\log n}{\log \gamma}$ is larger than the average overlap $p^2\binom{n}{2}$. Our proof in the
sparse case borrows some ideas from \cite{sudakov_similarity}. On the
other hand, one can not use their approach for all cases, hence some new ideas were needed
to prove Theorem~\ref{theorem_main}.

It will become evident from our proof that the problem of determining the discrepancy can be essentially reduced to the
following question. Let $K>0$ and let $X$ be a binomial random variable with parameters $m$ and $\rho$.
What is the maximum value of $\Lambda = \Lambda(m,\rho,K)$ satisfying
$\Prob\big[X-m\rho > \Lambda\big] \ge e^{-K}$? This question is related to the rate function of binomial distribution.
In all cases, the discrepancy in the statement of Theorem~\ref{theorem_main} is w.h.p.
\begin{equation}
\label{equation_true_answer}
\disc(G,H) = \Theta_k\left(n \cdot \Lambda\Big( p\binom{n-1}{k-1}, q, \log n\Big)\right).
\end{equation}
Note that $p\binom{n-1}{k-1}$ is roughly the size of the neighborhood of a vertex in the hypergraph $G$.

The rest of this paper is organized as follows. Section~\ref{section_concentration} contains a list of inequalities
and technical lemmas used throughout the paper. In section~\ref{section_upper}, we define the
{\it probabilistic discrepancy} $\discP(G,H)$ and prove that w.h.p. it does not deviate too much from $\disc(G,H)$.
Additionally, we establish the upper bounds for $\disc(G,H)$
based on analogous bounds for $\discP(G,H)$. In section~\ref{section_lower}, we give a detailed proof of the lower
bounds. The final section contains some concluding remarks and open problems. In this paper, the function $\log$ refers to the natural
logarithm and all asymptotic notation symbols ($\Omega$, $O$, $o$ and $\Theta$) are with respect to the variable $n$.
Furthermore, the $k$-subscripts in these symbols indicate the dependence on $k$ in the relevant constants.

\section{Auxiliary results}
\label{section_concentration}

In this section we list and prove some useful concentration inequalities about the binomial and hypergeometric
distributions and also prove a corollary from the well-known Vizing's Theorem which asserts the existence of a
linear-size matching in {\it nearly regular} graphs (i.e., the maximum degree is close to the average degree).
We will not attempt to optimize our constants, preferring rather to choose values which provide a simpler
presentation. Let us start with classical Chernoff-type estimates for the tail of the binomial
distribution (see, e.g., \cite{alonspencer}).

\begin{lemma}
\label{chernoff_ineq} Let $X=\sum_{i=1}^l X_i$ be the sum of independent zero-one random variables with average
$\mu = \E[X]$. Then for all non-negative $\lambda \le \mu$, we have
$\Prob[|X-\mu| > \lambda] \le 2e^{-\frac{\lambda^2}{4\mu}}$.
\end{lemma}

The following lower tail inequality (see \cite{alonspencer}) is due to Janson.

\begin{lemma}
\label{janson_ineq}
Let $A_1, A_2,..., A_l$ be subsets of a finite set $\Omega$, and let $R$ be a random subset of $\Omega$ for which
the events $r \in R$ are mutually independent over $r \in \Omega$. Define $X_j$ to be the indicator random variable
of $A_j\subset R$. Let $X=\sum_{j=1}^l X_j$, $\mu= \E[X]$, and $\Delta=\sum_{i\sim j} \E[X_i\cdot X_j]$, where
$i\sim j$ means that $X_i$ and $X_j$ are dependent (i.e., $A_i$ intersects $A_j $). Then for any $\lambda>0$,
$$\Prob[X\le \mu-\lambda]< e^{-\frac{\lambda^2}{2\mu+\Delta}}\,.$$
\end{lemma}

In the proof of the dense case of the main theorem we will need a lower bound for the tail of the hypergeometric distribution. To prove it
we use the following well-known estimates for the binomial coefficient.

\begin{proposition}
\label{binomial_ineq}
Let $H(p)=p\log p+(1-p)\log (1-p)$, then for any integer $m>0$ and real $p\in (0,1)$ satisfying $pm\in \mathbb{Z}$
we have
$$\frac{\sqrt{2\pi}}{e^2}\le \binom{m}{pm} \sqrt{mp(1-p)} e^{m H(p)}\le \frac{e}{2\pi}.$$
\end{proposition}
\begin{proof}
This can be derived from Stirling's formula $\sqrt{2\pi m} \left(\frac{m}{e}\right)^m\le m!\le
e \sqrt{m} \left(\frac{m}{e}\right)^m$.
\end{proof}

\begin{lemma}
\label{hypergeom_ineq}
Let $d_1$, $d_2$, $\Delta$ and $N$ be integers and $K$ be a real parameter such that $1 \le d_1,d_2 \le \frac{2N}{3}$,
$1 \le K \le \frac{d_1d_2}{100N}$ and $\Delta = \sqrt{\frac{d_1 d_2K}{N}}$. Then
$$\sum_{t\ge \frac{d_1d_2}{N}+\Delta} \frac{\binom{d_1}{t} \binom{N-d_1}{d_2-t}}{\binom{N}{d_2}}\ge e^{-40K}.$$
\end{lemma}
\begin{proof}
For convenience, we write $f(t)=\binom{d_1}{t} \binom{N-d_1}{d_2-t}/\binom{N}{d_2}$. In order to show the
desired lower bound of the hypergeometric sum, it suffices to prove that
$$f(t)\ge \frac{4e^{-40K}}{\sqrt{\frac{d_1d_2}{N}+\Delta}},$$
for every integer $t=\frac{d_1d_2}{N}+\theta \Delta$ with $1 \le \theta \le 2$. Indeed, to see this, note that
there are at least $\lfloor \Delta \rfloor \ge \frac{\Delta}{2}$
integers between $\frac{d_1d_2}{N}+\Delta$ and $\frac{d_1d_2}{N}+2 \Delta$ and
$$\Delta > \frac{1}{2} \sqrt{\Delta^2 + \Delta} \geq \frac{1}{2}\sqrt{\frac{d_1d_2}{N}+\Delta}.$$

Next we prove the bound for $f(t)$. For our choice of $\Delta$, the inequality $\Delta \le \frac{d_1}{15}$ is true since
$$\Delta=\sqrt{\frac{d_1d_2K}{N}}= d_1\sqrt{\frac{d_2}{N}\cdot \frac{K}{d_1}} \le d_1\sqrt{\frac{d_2}{N}\cdot \frac{d_2}{100N}} = \frac{d_1}{10} \cdot \frac{d_2}{N}
\le
\frac{d_1}{15}.$$
Similarly $\Delta \le \frac{d_2}{15}$. Let $x=\frac{d_2}{N}$, $y=\frac{\theta \Delta}{d_1}$ and
$z=\frac{\theta \Delta}{N-d_1}$. Then $t = (x+y)d_1$ and $d_2 - t = (x-z)(N-d_1)$. But $0 < x+y < 1$, because
$0<x \le \frac{2}{3}$ and $0<y \le \frac{2\Delta}{d_1} < \frac{1}{3}$. Furthermore, $0 < x-z < 1$, because
$\frac{z}{x} =\frac{\theta \Delta N}{d_2(N-d_1)}\le \frac{3\theta\Delta}{d_2} \le \frac{2}{5}$ and $x \le \frac{2}{3}$.
By Proposition~\ref{binomial_ineq}, we have
$$f(t) =\frac{\binom{d_1}{(x+y)d_1}\binom{N-d_1}{(x-z)(N-d_1)}}{\binom{N}{xN}}\ge\frac{4\pi^2}{e^5}\sqrt{R} e^{-L},$$
where $L= d_1\cdot H(x+y) + (N-d_1)\cdot H(x-z) - N\cdot H(x)$ and
$$R=\frac{x(1-x)N}{(x-z)(1-x+z)(x+y)(1-x-y)d_1(N-d_1)}\ge \frac{1}{(x+y)d_1} \ge \frac{1}{2}\cdot \frac{1}{\frac{d_1d_2}{N}+\Delta}.$$
Here we used the inequality $\theta\leq 2$ and the identity $(x+y)d_1=t=\frac{d_1d_2}{N}+\theta\Delta$.
Because $d_1y=(N-d_1)z=\theta\Delta$ and $\log (1+s)\le s$, we obtain

\begin{align*}
L&=d_1\left[(x+y)\log\left(1 + \frac{y}{x}\right) + (1-x-y)\log\left(1-\frac{y}{1-x}\right) \right]\\
 &~~~+(N-d_1)\left[(x-z)\log \left(1 - \frac{z}{x}\right) + (1-x+z)\log \left(1+\frac{z}{1-x}\right)\right]\\
&\le d_1\left[\frac{(x+y)y}{x} - \frac{(1-x-y)y}{1-x} \right] + (N-d_1)\left[-\frac{(x-z)z}{x} + \frac{(1-x+z)z}{1-x}\right] \\
&= \theta \Delta\cdot(y+z)\cdot\left(\frac{1}{x} + \frac{1}{1-x}\right) = \frac{\theta^2\Delta^2 N^3}{d_1(N-d_1)d_2(N-d_2)} \le 36K.
\end{align*}
Thus we always have $f(t)\ge \frac{4\pi^2}{\sqrt{2}e^5}\cdot \frac{e^{-36K}}{\sqrt{\frac{d_1d_2}{N}+\Delta}}\ge
\frac{4e^{-40K}}{\sqrt{\frac{d_1d_2}{N}+\Delta}}$, completing the proof.
\end{proof}

The next lemma will be used to prove the lower bound in the sparse case of Theorem~\ref{theorem_main},
and was inspired by an analogous result in \cite{sudakov_similarity}.

\begin{lemma}
\label{sparse1}
For positive integers $n$ and $k$, let $N={n-\frac{n}{k} \choose k-1}$,
$\frac{w(n)}{N} \le p \le q \le \frac12$ and suppose that $pqN \leq \frac{1}{30}\log n$. Define $\gamma = \frac{\log n}{pq N}$.
Let $N_1, \ldots, N_s \subseteq B$ be $s \geq n^{1/3}$ disjoint sets of size $(1+o(1))Np$, and consider the random set $B_q$, obtained by taking each element of $B$ independently with
probability $q$. Then w.h.p., there is an index $i$ for which
\begin{enumerate}
\item [(1)]
$|B_q \cap N_i| \geq \frac{\log n}{6 \log \gamma}$ if $pN \ge \frac{\log n}{5\log \gamma}$;
\item [(2)]
$N_i \subseteq B_q$ if $pN < \frac{\log n}{5\log \gamma}$.
\end{enumerate}
\end{lemma}

\begin{proof}
Let $t=\frac{\log n}{6\log \gamma}$. Clearly $1-q \geq e^{-3q/2}$ when $q \le 1/2$.
For a fixed index $i$, the probability that $|B_q \cap N_i| \geq t$ is at least ${|N_i| \choose t}q^t(1-q)^{|N_i|-t}$. Using the bounds
${a \choose b} \geq (\frac{a}{b})^b$ for $a\ge b$, and $\frac{1}{30}\log n \geq Npq=\frac{\log n}{\gamma}$, we obtain
\begin{eqnarray*}
{|N_i| \choose t}q^t(1-q)^{|N_i|-t} &\geq& \left(\frac{(1+o(1))Npq}{t}\right)^t e^{-2pqN} \geq
\left(\frac{5 \log \gamma}{\gamma}\right)^{\frac{\log n}{6 \log \gamma}}n^{-1/15}\\
&\geq& n^{-1/6} \cdot n^{-1/15} \geq n^{-0.3}.
\end{eqnarray*}
Hence the expected number of indices $i$ such that $|B_q \cap N_i| \geq t$ is at least $s n^{-0.3} \geq n^{1/30}$.
Since the sets $N_i$ are disjoint, these events are independent for different choices of $i$. Therefore by Lemma~\ref{chernoff_ineq}
w.h.p. we can find such an index (actually many).

If $pN < \frac{\log n}{5\log \gamma}$, then $q=\frac{\log n}{\gamma pN} \geq \frac{5 \log \gamma}{\gamma} \geq \gamma^{-1}$.
Therefore the probability that some $N_i \subseteq B_q$ is
$$q^{|N_i|} \geq \gamma^{-(1+o(1)Np} \geq \gamma^{-\frac{\log n}{4\log \gamma}}=n^{-1/4},$$
and we can complete the proof as in the first case.
\end{proof}

The last lemma in this section, which can be easily derived from Vizing's Theorem, will be used to find a linear-size matching in nearly regular graphs.

\begin{lemma}
\label{lemma_matching}
Every graph $G$ with maximum degree $\Delta(G)$, contains a matching of size at least $\frac{e(G)}{\Delta(G) + 1}$.
\end{lemma}
\begin{proof}
By Vizing's Theorem, the graph $G$ has a proper edge coloring $f: E(G) \to \{1,2,\ldots, \Delta(G) + 1\}$. For each color $1 \le c \le \Delta(G) + 1$, the edges $f^{-1}(c)$ form a
matching in $G$. By the pigeonhole principle, there is a color $c$ such that $f^{-1}(c)$ has at least $\frac{e(G)}{\Delta(G) + 1}$ edges.
\end{proof}

\section{Upper bounds}
\label{section_upper}
In this section we prove the upper bound for the discrepancy in Theorem~\ref{theorem_main}.
Let $G$ and $H$ be two random hypergraphs over the same vertex set $V$, distributed according to $\RandHyp{k}{n}{p}$ and $\RandHyp{k}{n}{q}$, respectively.
The \textit{probabilistic discrepancy} of $G$ and $H$ is defined by
$$\discP(G,H) = \max_{\pi} \left|e(G_\pi \cap H) - pq \binom{n}{k}\right|,$$
where the maximum is taken over all bijections $\pi: V \to V$. We will show that w.h.p. the difference between $\disc(G,H)$ and $\discP(G,H)$ is very small.
Before we proceed, we state the following fact whose proof is fairly trivial.
\begin{proposition}
\label{proposition_XY}
If $|AB-A_0B_0|\ge \epsilon_1\epsilon_2+|A_0|\epsilon_2+|B_0|\epsilon_1$, then either $|A-A_0|\ge \epsilon_1$ or $|B-B_0|\ge \epsilon_2$.
\end{proposition}
\begin{lemma}
\label{lemma_prob_disc}
With probability at least $1 -4e^{-\sqrt{n}}$, the inequality $|\disc(G,H) - \discP(G,H)| \le 2\varepsilon$ holds,
where $\varepsilon=4n^{\frac{1}{4}}\sqrt{pq\binom{n}{k}}$.
\end{lemma}
\begin{proof}
Since $p\binom{n}{k} = \Omega(n)$, applying Lemma~\ref{chernoff_ineq} to the random variable $e(G)$ for $\lambda=2n^{\frac{1}{4}}\sqrt{p\binom{n}{k}}\le p\binom{n}{k}$
yields
$$\Prob\left[\Big|e(G)-p\binom{n}{k}\Big|\le 2n^{\frac{1}{4}}\sqrt{p\binom{n}{k}}~\right]\ge 1-2e^{-\sqrt{n}}.$$
Similarly, we have $\Prob\left[|e(H)-q\binom{n}{k}|\le 2n^{\frac{1}{4}}\sqrt{q\binom{n}{k}}~\right]\ge 1-2e^{-\sqrt{n}}$.
Therefore, with probability at least $1-4e^{-\sqrt{n}}$, $|\rho_{G}-p|\le 2n^{\frac{1}{4}} \big(p/\binom{n}{k}\big)^{1/2}$ and
$|\rho_{H}-q|\le 2n^{\frac{1}{4}} \big(q/\binom{n}{k}\big)^{1/2}$. These inequalities, together with Proposition~\ref{proposition_XY}, imply
$$\left|\rho_{G}\rho_H\binom{n}{k}-pq\binom{n}{k}\right|\le 4\sqrt{pqn} + 2pn^{\frac{1}{4}}\sqrt{q\binom{n}{k}} + 2qn^{\frac{1}{4}}\sqrt{p\binom{n}{k}}\le 2\varepsilon,$$
completing the proof of the lemma.
\end{proof}
It is easy to check that the error term $\varepsilon$ is much smaller than the bounds in Theorem~\ref{theorem_main}.
Therefore, in order to prove Theorem~\ref{theorem_main} for $\disc(G,H)$, it suffices to prove the corresponding bounds for $\discP(G,H)$ instead.

\begin{lemma}
Let $G$ and $H$ be as in Theorem~\ref{theorem_main}. Then with high probability $\discP(G,H)$
satisfies the stated upper bounds of this theorem.
\end{lemma}
\begin{proof}
Since the number of edges of $G$ is distributed binomially and $p\binom{n}{k}=\Omega(n)$, by Lemma~\ref{chernoff_ineq},
we have $e(G)< 2p\binom{n}{k}$ with probability at least $1-e^{-\Theta(n)}$. Since $\discP(G,H)$ is bounded by $\max\big\{e(G),pq\binom{n}{k}\big\}$, this implies the
assertion in the case (2.2) of Theorem~\ref{theorem_main}.

For any fixed bijection $\pi:V\to V$, the number of edges in $G_\pi \cap H$ is distributed binomially with parameters $\binom{n}{k}$ and $pq$.
If $pq\binom{n}{k} > 4n\log n$ let $\lambda=2\sqrt{pq\binom{n}{k}n\log n}\leq pq\binom{n}{k}$. Then by Lemma~\ref{chernoff_ineq}, the probability that
$\left|e(G_\pi \cap H) - pq\binom{n}{k}\right|> \lambda$ is at most $2e^{-n\log n}$. On the other hand, if
$pq\binom{n}{k}\le 4n\log n$, let $\gamma' = 4e\frac{n \log n}{pq \binom{n}{k}}\geq e$
and $\lambda=\frac{4e^2 n \log n}{\log \gamma'}\geq \frac{4e^2 n \log n}{\gamma'} = epq \binom{n}{k}$.
Since $\binom{a}{b} \le \left(\frac{ea}{b}\right)^b$, the probability that $e(G_\pi \cap H)> \lambda$ is at most
$${\binom{n}{k} \choose \lambda}(pq)^{\lambda} \leq \left(\frac{e\binom{n}{k}pq}{\lambda}\right)^{\lambda}=
\left(\frac{4e^2n \log n}{\gamma' \lambda}\right)^{\lambda} = \left(\frac{\gamma'}{\log \gamma'}\right)^{-\frac{4e^2 n \log n}{\log \gamma'}}<e^{-n\log n}.$$
Since there are $n!$ possible bijections $\pi:V\to V$, by the union bound
$$\Prob\left[\discP(G,H)>\lambda~\right]\le n!\cdot 2e^{-n\log n}\le e^{-n/2}.$$
To finish the proof of the lemma note that $\gamma$, defined in Theorem~\ref{theorem_main}, satisfies
$\gamma=\Theta_k(\gamma')$. Also observe that for $p,q$ satisfying both $pq\binom{n}{k}\le 4n\log n$ and $pqN \geq \frac{1}{30}\log n$, where $N={n-\frac{n}{k} \choose k-1}$, we have
$\sqrt{pq\binom{n}{k}n\log n}=\Theta_k\left( \frac{4e^2 n \log n}{\log \gamma'}\right)$.
\end{proof}

\section{Lower bounds}
\label{section_lower}

In this section we prove the lower bounds in Theorem~\ref{theorem_main}. As we previously explained, it is enough to obtain these bounds for
$\discP(G,H)$. We divide the proof into two cases.
The first ({\em dense case}) will be discussed in the next subsection.
The second ({\em sparse case}) will be discussed in
subsection~\ref{subsection_sparse}. Throughout the proofs, we assume that $k$ is fixed and $n$ is tending to infinity.

\subsection{Dense Case}\label{subsection_framework}
Let $N=\binom{n-\frac{n}{k}}{k-1}$ and let $p,q$ be such that $pq N>\frac{1}{30}\log n$.
Select an arbitrary set $L \subseteq V$ of size $|L| = \frac{n}{k}$.
We prove that w.h.p. there exists an \textit{$L$-bijection} $\pi:V\to V$ with overlap
\begin{equation}
\label{inequality_high_deviation}
e(G_\pi \cap H) \ge pq\binom{n}{k} + \Theta_k\left(n \cdot \sqrt{pqN \log n}\right)=pq\binom{n}{k} + \Theta_k\left(\sqrt{pq\binom{n}{k}n \log n}\right),
\end{equation}
where an $L$-bijection $\pi:V\to V$ is a bijection from $V$ to $V$ which only permutes the elements of $L$, i.e., $\pi(x) = x$ for all $x\not\in L$.

From the random hypergraph $G$ we construct a random bipartite graph $\widetilde{G}$ with vertex set $L_G\cup R$, where $L_G = L$ and $R$ is the set of
all $(k-1)$-tuples in $V\setminus L$. Note that $|R|=N$.
The vertices $v_1\in L_G$ and $\{v_2,v_3,\ldots, v_k\}\in R$ are adjacent if $\{v_1,v_2,\ldots, v_k\}$ forms an edge in the
hypergraph $G$. With slight abuse of notation, we view $\widetilde{G}$ as a sub-hypergraph of $G$,
containing all edges $e$ having exactly one vertex in $L$, i.e. $|e\cap L|=1$.
Similarly, from the random hypergraph $H$ we construct a random bipartite graph $\widetilde{H}$ with vertex set $L_H\cup R$.
Figure~\ref{picture_graphs} shows the resulting bipartite
graphs.

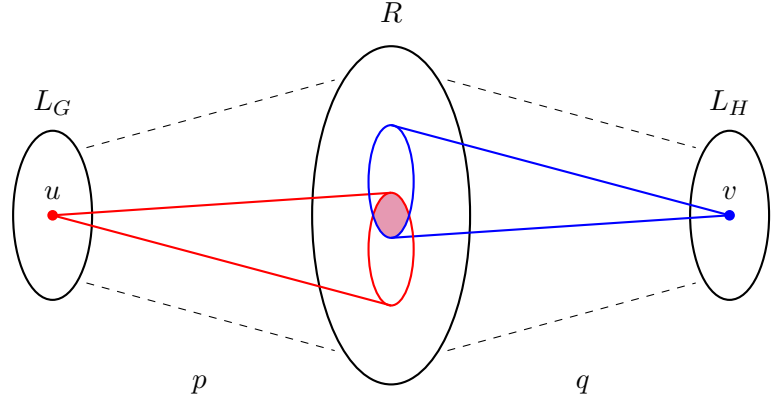
\begin{figure}[ht!]
  \centering
  \begin{tikzpicture}[scale=1.5,auto=left]
    \begin{scope}
      \clip (0,-0.3) ellipse (0.2 and 0.5);
      \clip (0,0.3) ellipse (0.2 and 0.5);
      \fill[color=purple!40] (-1,-1) rectangle (1,1);
    \end{scope}
    \draw[thick] (-3, 0) ellipse (0.35 and 0.75);
    \draw[thick] (0, 0) ellipse (0.7 and 1.5);
    \draw[thick] (3, 0) ellipse (0.35 and 0.75);
    \node at (-3, 1.0) {$L_G$};
    \node at (0, 1.8) {$R$};
    \node at (3, 1.0) {$L_H$};
    \node[circle, fill=red, inner sep=0pt, minimum width=4pt] at (-3, 0) [label=north:$u$]{};
    \node[circle, fill=blue, inner sep=0pt, minimum width=4pt] at (3, 0) [label=north:$v$]{};
    \draw[color=red, thick] (-3, 0) -- (0.0, 0.2);
    \draw[color=red, thick] (-3, 0) -- (0.0, -0.8);
    \draw[color=blue, thick] (3, 0) -- (0.0, 0.8);
    \draw[color=blue, thick] (3, 0) -- (0.0, -0.2);
    \draw[color=red, thick] (0,-0.3) ellipse (0.2 and 0.5);
    \draw[color=blue, thick] (0,0.3) ellipse (0.2 and 0.5);
    \draw[dashed] (-2.7, 0.6) -- (-0.5, 1.2);
    \draw[dashed] (-2.7, -0.6) -- (-0.5, -1.2);
    \draw[dashed] (0.5, 1.2) -- (2.7, 0.6);
    \draw[dashed] (0.5, -1.2) -- (2.7, -0.6);
    \node at (-1.7, -1.5) {$p$};
    \node at (1.7, -1.5) {$q$};
  \end{tikzpicture}
  \caption{Random bipartite graphs $\widetilde{G}$ and $\widetilde{H}$.}
  \label{picture_graphs}
\end{figure}

Given an $L$-bijection $\pi:V\to V$, we divide the edge set of $G_\pi\cap H$ into two subsets: the edge set of $\widetilde{G}_\pi\cap \widetilde{H}$ and its complement.
To prove our result we first expose the random edges in $\widetilde{G}$ and $\widetilde{H}$, and show how to find an $L$-bijection $\pi$ having overlap
at least $\Theta_k\left(n \cdot \sqrt{pqN \log n}\right)$ more than the expectation. Then we fix such $\pi$ and expose all the remaining edges in
$G$ and $H$ showing that the contribution of these edges to $G_\pi\cap H$ does not deviate much from the expected contribution.
More precisely, let $e_\pi= |E((G-\widetilde{G})_\pi)\cap E(H-\widetilde{H})|$, then $e(G_\pi\cap H)=e(\widetilde{G}_\pi\cap \widetilde{H})+ e_\pi$.
Moreover, $e_\pi$ is distributed according to $\Bin(m,pq)$, where $\frac{1}{2}\binom{n}{k} \leq m=\binom{n}{k}-N\frac{n}{k} \le \binom{n}{k}$. Thus w.h.p.
$|e_\pi -pqm| < \sqrt{pqm}\cdot\log n$, as Lemma~\ref{chernoff_ineq} shows. Since $\sqrt{pqm}\cdot \log n\ll \sqrt{pq\binom{n}{k}n \log n}$, in order to obtain (\ref{inequality_high_deviation}),
it is enough to show that w.h.p. there exists an $L$-bijection $\pi$ such that
\begin{equation}
\label{inequality_bipartite_graphs}
e(\widetilde{G}_\pi\cap \widetilde{H})\ge \frac{n}{k}\cdot\left(pqN+\Theta_k\left(\sqrt{pqN \log n}\right)\right).
\end{equation}

We define an auxiliary bipartite graph $\Gamma = \Gamma(\widetilde{G},\widetilde{H})$ as follows. A vertex $u\in  L_G$ {\it survives} if
$|\deg_{\widetilde{G}}(u)-pN|\le 2\sqrt{2pN}$ and similarly, a vertex $v\in L_H$ {\it survives} if $ |\deg_{\widetilde{H}}(v) - qN| \le 2 \sqrt{2qN}$. Let $S_G$ and $S_H$ be the sets of all surviving vertices of
$\widetilde{G}$ and $\widetilde{H}$, respectively. Let $s_G=|S_G|$ and $s_H=|S_H|$. The set of vertices of $\Gamma$ is the union of $S_G$ and $S_H$. The edges of $\Gamma$ are defined by the
property
$$u\sim_{\Gamma} v \iff \codeg(u,v) \ge \frac{\deg_{\widetilde{G}}(u)\deg_{\widetilde{H}}(v)}{N} + 10^{-2} \sqrt{pqN \log n},$$
where $\codeg(u,v)$ denotes the {\it codegree} of $u\in L_G$ and $v\in L_H$, i.e. $\codeg(u,v)=|N_{\widetilde{G}}(u)\cap N_{\widetilde{H}}(v)|$.
The graph $\Gamma$ has many vertices in both parts, as the following simple lemma demonstrates
\begin{lemma}
\label{lemma_survival}
W.h.p. each part of $\Gamma$ has size at least $\frac{n}{4k}$.
\end{lemma}
\begin{proof}
Let $\alpha$ be the probability that some vertex $u$ survives in $L_G$. Since $pN \ge w(n) \ge 8$, we have that
$2\sqrt{2pN} \leq pN$. Thus Lemma~\ref{chernoff_ineq} applied to $\deg_{\widetilde{G}}(u)$ implies $\alpha \ge 1-2e^{-2}\ge 1/2$.
Since the events that vertices survive are independent, $s_G$ stochastically dominates the binomial distribution with
parameters $n/k$ and $1/2$. Thus, again by Lemma~\ref{chernoff_ineq}, w.h.p. $s_G \geq n/(4k)$ and a similar estimate holds for $s_H$.
\end{proof}

To prove \eqref{inequality_bipartite_graphs}, we will show that the following two statements hold w.h.p.
\begin{itemize}
\item[(a)] $\Gamma$ has a matching $M = \{(u_1, v_1), \ldots, (u_l, v_l)\}$ of size $l=\frac{n}{50k}$;
\item[(b)] there exists an $L$-bijection $\pi$ such that $\pi(u_i)=v_i$ for
all $i=1,2,\ldots, l,$ and, $$ \sum_{u \in L_G\setminus\{u_1,u_2,\ldots, u_l\}} \codeg(u, \pi(u)) \ge \left(\frac{n}{k}-l\right)pqN - 2\frac{n}{k}\sqrt{pqN}.$$
\end{itemize}
Indeed, for any two adjacent vertices $u,v$ in $\Gamma$, we have
$$\frac{\deg_{\widetilde{G}}(u)\deg_{\widetilde{H}}(v)}{N}\ge \frac{(pN-\sqrt{8pN})(qN-\sqrt{8qN})}{N} \ge pqN - 6\sqrt{pqN}.$$
Thus using (a), (b) and $l = \frac{n}{50k}$ we obtain
\begin{eqnarray*}
e(\widetilde{G}_\pi\cap \widetilde{H}) &=& \sum_{u \in L_G} \codeg(u, \pi(u)) \geq \sum_{i=1}^l\codeg(u_i, v_i)+
\left(\frac{n}{k}-l\right)pqN - 2\frac{n}{k}\sqrt{pqN}\\
&\geq& \sum_{i=1}^l\left[ \frac{\deg_{\widetilde{G}}(u_i)\deg_{\widetilde{H}}(v_i)}{N} + 10^{-2} \sqrt{pqN \log n}\right]+\left(\frac{n}{k}-l\right)pqN -
2\frac{n}{k} \sqrt{pqN}\\
&\geq&  \sum_{i=1}^l \left[ pqN - 6\sqrt{pqN}\right] +\frac{n}{50k}10^{-2} \sqrt{pqN \log n}+\left(\frac{n}{k}-l\right)pqN -
2\frac{n}{k}\cdot \sqrt{pqN}\\
&\geq& \frac{n}{k} \left( pqN + 10^{-4} \sqrt{pqN \log n}\right)
\end{eqnarray*}

We need the following lemma in order to prove that (b) holds.
\begin{lemma}\label{average}
Let $0<\alpha<1$ be any absolute constant. Then with probability at least
$1-e^{-\frac{n}{k}}$, any two subsets $A\subseteq L_G$ and $B\subseteq L_H$ with $|A|=|B|=\frac{\alpha n}{k}$ satisfy $$X_{A,B}:=\sum_{u\in A, v\in B} \codeg(u,v)\ge
\left(\frac{\alpha n}{k}\right)^2pqN- 2\alpha \left(\frac{n}{k}\right)^2\sqrt{pqN}.$$
\end{lemma}

\begin{proof} Let $X_{w,u,v}$ be the indicator of $wu\in E(\widetilde{G})$ and $wv\in E(\widetilde{H})$ for $w\in R, u\in A, v\in B$. So $X_{A,B}=\sum_{w\in R, u\in A, v\in B}
X_{w,u,v}$ and $\E[{X_{w,u,v}}]=pq$. Moreover, $X_{w,u,v}$ and $X_{w',u',v'}$ are dependent if and only if $wu=w'u'$ or $wv=w'v'$. Thus, $\mu=\E[X_{A,B}]= \left(\frac{\alpha
n}{k}\right)^2 N pq$ and $$\Delta=\sum_{w\in R, u\in A}~ \sum_{v,v'\in B} \E[X_{w,u,v}\cdot X_{w,u,v'}]+\sum_{w\in R, v\in B}~ \sum_{u,u'\in A} \E[X_{w,u,v}\cdot
X_{w,u',v}]=\frac{\alpha n}{k} \binom{\frac{\alpha n}{k}}{2} Npq\left(p+q\right),$$
where $\mu$ and $\Delta$ are defined as in Lemma~\ref{janson_ineq}.
Let $F$ be the event that there exists at least one pair of subsets $A\subseteq L_G, B\subseteq L_H$ with $|A|=|B|=\frac{\alpha n}{k}$ satisfying $X_{A,B}< (\frac{\alpha n}{k})^2 Npq
-2\alpha (\frac{n}{k})^2 \sqrt{Npq}$. By the union bound and by Lemma~\ref{janson_ineq}, we have
\begin{eqnarray*}
\Prob[F]&\le& \sum_{A\in \binom{L_G}{\alpha n}, B\in \binom{L_H}{\alpha n}} \Prob\left[X_{A,B}<\mu -2\alpha \left(\frac{n}{k}\right)^2 \sqrt{Npq}\right]
\le \binom{\frac{n}{k}}{\frac{\alpha n}{k}}^2 e^{- \frac{\left(2\alpha (\frac{n}{k})^2 \sqrt{Npq}\right)^2}{2\mu+\Delta}}\\
&\le& \left(\frac{e}{\alpha}\right)^{\frac{2\alpha n}{k}} e^{-3\frac{n}{k}} \le e^{-\frac{n}{k}},
\end{eqnarray*}
since $2\mu + \Delta \le \frac{4}{3}\left(\frac{\alpha n}{k}\right)^3Npq$, $\alpha<1$ and $\alpha \log (e/\alpha) \leq 1$ for all such $\alpha$.\qedhere
\end{proof}
Let $M = \{(u_1, v_1),\ldots, (u_l, v_l)\}$ be a matching satisfying (a) and  let $A=L_G\setminus\{u_1,u_2,\ldots, u_l\}$ and $B=L_H\setminus\{v_1,v_2,\ldots, v_l\}$.
One can write $|A|=|B|=\frac{n}{k}-l=\frac{\alpha n}{k}$, where $\alpha=\frac{49}{50}$. Consider $X_{A,B}=\sum_{u\in A, v\in B} \codeg(u,v)$.
Then, by Lemma~\ref{average}, with probability at least $1-e^{-\frac{n}{k}}$, we have
$$\sum_{u\in A, v\in B} \codeg(u,v)\ge \left(\frac{n}{k}-l\right)^2pqN- 2 \frac{n}{k}\left(\frac{n}{k}-l\right)\sqrt{pqN}.$$
Since the complete bipartite graph with parts $A,B$ is a disjoint union of $\frac{n}{k}-l$ perfect matchings,
by the pigeonhole principle, there exists a matching $M'$ between $A$ and $B$ such that
$$\sum_{(u,v)\in M'} \codeg(u,v) \ge \frac{\sum_{u\in A, v\in B} \codeg(u,v)}{\frac{n}{k}-l}\ge \left(\frac{n}{k}-l\right)pqN - \frac{2n}{k}\sqrt{pqN}.$$
Then the matching $M\cup M'$ between $L_G$ and $L_H$ gives the desired $L$-bijection $\pi$ and proves (b).

\medskip

To finish the proof we need to establish (a). If $\Gamma$ is nearly regular, then by Lemma~\ref{lemma_matching}, $\Gamma$ would contain a linear-size matching.
Unfortunately this is not the case. However, we will show that it is possible to delete some edges of $\Gamma$ at random and obtain a
\textit{pruned graph} $\Gamma'$, which is nearly regular. Let
$$f(d_1, d_2):= \Prob\left[u \sim_\Gamma v | \deg_{\widetilde{G}}(u) = d_1, \deg_{\widetilde{H}}(v)=d_2\right],$$
where $|d_1 - pN| \le 2\sqrt{2pN}$ and $|d_2 - qN| \le 2\sqrt{2qN}$. Let $f_0$ be the minimum of $f(d_1,d_2)$ over all pairs $(d_1,d_2)$
in the domain of $f$. Suppose that $f_0\ge n^{-\frac{1}{2}}$, which we shall prove later.
We keep each edge $uv$ of $\Gamma$ in $\Gamma'$ independently with probability
$\frac{f_0}{f(d_1,d_2)}$, where $d_1 = \deg_{\widetilde{G}}(u)$ and $d_2= \deg_{\widetilde{H}}(v)$. Then, we claim that
for any vertex $u\in S_G$, $\deg_{\Gamma'}(u)$ is binomially distributed with parameters $s_H$ and $f_0$.
Indeed, by definition, $\Prob\left[u \sim_{\Gamma'} v | \deg_{\widetilde{G}}(u) = d_1,
\deg_{\widetilde{H}}(v)=d_2\right] = f_0$ for all possible $d_1,d_2$. Moreover, conditioning on the neighbors of $u$ in $\widetilde{G}$ and on the values of the degrees
$\deg_{\widetilde{H}}(v_1)$, $\deg_{\widetilde{H}}(v_2)$, $\ldots,$ $\deg_{\widetilde{H}}(v_m)$, the events $u\sim_{\Gamma} v_1, u\sim_{\Gamma} v_2$, $\ldots,$ and $u\sim_{\Gamma} v_m$
are all independent. Therefore, by definition of $\Gamma'$, it is easy to see that $u\sim_{\Gamma'} v_1$, $u\sim_{\Gamma'} v_2$, $\ldots$, and
$u\sim_{\Gamma'} v_m$ are independent as well. Thus for any $u\in S_G$, $\deg_{\Gamma'}(u)\sim \Bin(s_H,f_0)$ and similarly, $\deg_{\Gamma'}(v)\sim \Bin(s_G,f_0)$
for all $v\in S_H$.

Conditioning on the degrees of all vertices in $\widetilde{G}, \widetilde{H}$, we obtain sets $S_G$ and $S_H$, which w.h.p. satisfy the assertion
of Lemma~\ref{lemma_survival}, i.e., $|S_G|=s_G\geq \frac{n}{4k}$ and $|S_H|=s_H\geq \frac{n}{4k}$. Thus both $s_Gf_0$ and $s_Hf_0$ are $\Omega_k(\sqrt{n})$.
Since all degrees in $\Gamma'$ are binomially distributed, Lemma~\ref{lemma_survival} together with the union bound imply that w.h.p. all vertices $u\in S_G, v\in S_H$ satisfy
$$\frac{s_Hf_0}{2}\le \deg_{\Gamma'}(u)\le \frac{3s_Hf_0}{2} ~~\mbox{and}~~\frac{s_Gf_0}{2}\le \deg_{\Gamma'}(v)\le \frac{3s_Gf_0}{2}.$$
Therefore, the max-degree $\Delta(\Gamma')\le \max\left\{\frac{3s_Hf_0}{2},\frac{3s_Gf_0}{2}\right\}\le \frac{3nf_0}{2k}$ and $e(\Gamma')\ge \frac{s_Gs_Hf_0}{2}\geq
\frac{n^2f_0}{32k^2}$. Thus by Lemma~\ref{lemma_matching}, $\Gamma'$ has a
matching of size at least $\frac{e(\Gamma')}{\Delta(\Gamma')+1}\ge \frac{n}{50k},$ completing the proof of (a).

It remains to prove the bound $f_0\ge n^{-\frac{1}{2}}$. Let $K=\frac{\log n}{5000}\ge 1$.
Since $pN$ tends to infinity, $p \leq q \leq 1/2$ and  $|d_1- pN| \le 2\sqrt{2pN}$, we have
$1 \le  d_1=(1+o(1)) pN \le \frac{2N}{3}$. Similarly $1 \le  d_2 =(1+o(1)) qN \le \frac{2N}{3}$.
Also recall that $pqN \geq \frac{1}{30}\log n$, which implies
$$\frac{d_1d_2}{100N}=(1+o(1))\frac{pqN}{100} \geq (1+o(1))\frac{\log n}{3000}>K.$$
Therefore we can apply Lemma~\ref{hypergeom_ineq}
with $\Delta=\sqrt{\frac{d_1d_2K}{N}}>\frac{\sqrt{pqN\log n}}{100}$. By the definition of $f(d_1,d_2)$, we have
$$ f(d_1,d_2) = \sum_{t\ge \frac{d_1d_2}{N} + \frac{\sqrt{pqN\log n}}{100}}\frac{\binom{d_1}{t}\binom{N-d_1}{d_2-t}}{\binom{N}{d_2}} \ge \sum_{t\ge
\frac{d_1d_2}{N}+\Delta} \frac{\binom{d_1}{t}\binom{N-d_1}{d_2-t}}{\binom{N}{d_2}} \geq e^{-40K} >n^{-\frac{1}{2}}.$$
This completes the proof. \qed

\subsection{Sparse case}
\label{subsection_sparse}

In this subsection, we prove the lower bound in the sparse case $pq N \leq\frac{1}{30}\log n$. Note that, since $p \leq q$ in this case,
we have $p \leq N^{-1/2+o(1)}$. The proof runs along the same lines as that of the dense case differing only in the application of Lemma~\ref{sparse1}
to obtain an $L$-bijection $\pi:V\to V$ whose sum of codegrees $\sum_{u \in L_G} \codeg(u, \pi(u))$ is large.
Suppose first that $pN \geq \frac{\log n}{5\log \gamma}$. Recall that $\gamma = \frac{\log n}{pq N} \geq 30$ and thus
$\frac{\log n}{6\log \gamma} \geq \frac{\log n}{42\log \gamma}+\frac{\log n}{\gamma}=\frac{\log n}{42\log \gamma}+pq N$.
Therefore it is enough to find a bijection $\pi$ between $L_G$ and $L_H$ such that
$\sum_{u \in L_G} \codeg(u, \pi(u)) \geq (1+o(1))\frac{n}{k}\cdot \frac{\log n}{6\log \gamma}$.

Partition the vertices of $L_G$ into $r=\frac{n}{ks}$ disjoint sets $S_1, \ldots, S_r$ each of size $s=n^{2/5}$. We will construct $\pi$
by applying the following greedy algorithm to each set. Let us start with $S_1$.
The algorithm will reveal the edges emanating from $S_1$ to $R$ in $\widetilde{G}$ by repeatedly exposing the neighborhood of a vertex in $S_1$, one at a time.
Throughout this process, we construct a subset $S_1' \subseteq S_1$ of size $(1-o(1))|S_1|$ and a family of disjoint sets $N_u\subseteq R$, such that each
$N_u$ has size $(1 + o(1)) Np$ and is contained in the neighborhood of $u$, for all $u\in S_1'$.
At each step, we pick a fresh vertex $u$ in $S_1$ and expose its neighborhood. If $u$ has a set of $(1+o(1))Np$ neighbors which is disjoint from $N_{w}$ for all $w$ in the current $S_1'$, denote this
particular set by $N_u$ and put $u$ in the set $S_1'$; otherwise move to the next step. At every step, the union $X = \cup_{w\in S_1'} N_w$ has size at most $O(pN\cdot s) \leq N^{0.9+o(1)}$.
Moreover, every vertex in $R\setminus X$ is adjacent to $u$ independently with probability $p$.
Since $pN\ge w(n)$ tends to infinity with $n$, the set of neighbors of $u$ outside $X$ has size $(1+o(1))|R\setminus X|p=(1+o(1))Np$
with probability $1-o(1)$. Furthermore, for different vertices such events are independent. Therefore, by Lemma~\ref{chernoff_ineq}, w.h.p. $|S_1'|=(1-o(1))|S_1|$.
Now we will construct the partial matching for $S_1$. Consider the disjoint sets $N_u$, for $u \in S_1'$, each of size
$(1+o(1))Np$. Pick an arbitrary vertex $v$ in $L_H$ and expose its neighbors in $\widetilde{H}$. This is a random subset $N_v$ of $R$, obtained by taking each element
independently with probability $q$. Therefore by case (1) of Lemma~\ref{sparse1}, w.h.p there is a vertex $u \in S_1'$ such that
$\codeg(u, v) \geq |N_u \cap N_v| \geq \frac{\log n}{6\log \gamma}$. Define $\pi(u)=v$, remove $u$ from $S_1'$, remove $v$ from $L_H$ and continue. Note that, as long
as there are at least $n^{1/3}$ vertices remaining in $S'_1$, we can match one of them with a newly exposed vertex from $L_H$ such that the codegree of this pair
is at least $\frac{\log n}{6\log \gamma}$. Once the number of vertices in $S_1'$ drops below $n^{1/3}$, leave the remaining vertices unmatched. W.h.p. we can match
a $1-o(1)$ fraction of the vertices in $S_1$.

Continue the above procedure for $S_2, \ldots, S_r$ as well. At the end of the process, we will have matched a $1-o(1)$ fraction of all the vertices in $L_G$
with distinct vertices in $L_H$ such that codegree of every matched pair is at least $\frac{\log n}{6\log \gamma}$. Therefore the sum of the codegrees of
this partial matching is at least $(1+o(1))\frac{n}{k}\cdot \frac{\log n}{6\log \gamma}$. To obtain the bijection $\pi$, one can match the remaining vertices in $L_G$
and $L_H$ arbitrarily.

When  $pN < \frac{\log n}{5\log \gamma}$ the same proof as above together with case (2) of Lemma~\ref{sparse1}
yields a bijection $\pi$ such that $\sum_{u \in L_G} \codeg(u, \pi(u)) \geq (1+o(1))\frac{n}{k}\cdot pN$. Since $ q\leq \frac12$, this is at least
$\left(\frac12+o(1)\right)\frac{n}{k}\cdot pN$ more than the expectation, finishing the analysis of the sparse case.
\qed

\section{Concluding remarks}
\label{section_remarks}
As we stated in the introduction, Theorem~\ref{theorem_main} also yields tight bounds when $p$ and/or $q>\frac{1}{2}$.
For any $G$ and $H$, one can check that $\disc(G, \overline{H}) = \disc(G,H)$, where $\overline{H}$ is the complement of $H$.
Moreover, $\overline{H}$ is distributed according to $\RandHyp{k}{n}{1-q}$, hence we can reduce the case $q>\frac{1}{2}$ to the case
$q'=1-q\le \frac{1}{2}$; the same holds when we take the complement of $G$ instead. We remark that one can determine
the discrepancy when $p$ is smaller than $\frac{\omega(n)}{N}$, but we chose not to discuss this range here, since the proof is similar to
the sparse case and it wouldn't provide any new insight.

The definition of discrepancy can be rephrased as $\disc(G,H)=\max\,\{\disc^+(G,H), \disc^-(G,H)\}$, where
$\disc^+(G,H)=\max_\pi \,e(G_\pi\cap H)-\rho_G\rho_H\binom{n}{k}$ and  $\disc^-(G,H)=\rho_G\rho_H\binom{n}{k}-\min_\pi \,e(G_\pi\cap H)$ are
the \textit{one-sided relative discrepancies}. In fact, all the lower bounds we obtained are for $\disc^+(G,H)$, and some of them
are not true for $\disc^-(G,H)$. This is because $\disc^-(G,H) \le \rho_G \rho_H \binom{n}{k} \simeq pq \binom{n}{k}$ and in
the sparse case, $pq\binom{n}{k}$ could be much smaller than $\disc(G,H)$.
Under the same hypothesis and using similar ideas as in Theorem~\ref{theorem_main}, one can show that
$$\disc^{-}(G,H) =\left\{ \begin{array}{ll}
\Theta_k\left(\sqrt{pq\binom{n}{k} n \log n}\right) & \text{ if } pqN > \frac{1}{30} \log n; \\
\Theta_k\left(pq\binom{n}{k}\right) & \text{ otherwise. }
\end{array}\right.$$
The last equation is related to the lower tail of the binomial distribution.

Lastly, we would like to mention that there are a substantial number of open problems about $\disc(G,H)$ and its related topics in \cite{bollobas_inter}.

\end{document}